\documentclass[11pt]{amsart}
\usepackage{latexsym}
\usepackage{amsthm}
\usepackage{amsmath}
\usepackage{amssymb}
\usepackage{color}
\usepackage{url}
\usepackage{multirow}
\allowdisplaybreaks
\usepackage{romanbar}

\theoremstyle{theorem}
\newtheorem{theorem}{Theorem}
\newtheorem{conjecture}{Conjecture}
\newtheorem{proposition}{Proposition}
\newtheorem{corollary}{Corollary}
\newtheorem{lemma}{Lemma}

\theoremstyle{definition}
\newtheorem*{definition}{Definition}
\newtheorem*{remark}{Remark}
\newtheorem*{example}{Example}

\newcommand{\F}{\mathbb{F}}
\newcommand{\C}{\mathbb{C}}

\newcommand{\Z}{\mathbb{Z}}
\newcommand{\N}{\mathbb{N}}
\newcommand{\Q}{\mathbb{Q}}
\newcommand{\cF}{{\mathcal F}}
\newcommand{\bea}{\begin{eqnarray*}}
\newcommand{\eea}{\end{eqnarray*}}
\newcommand{\mb}{{\rm MB}}
\newcommand{\mfb}{{\rm MFB}}

\begin{document}

\title{Polynomials in Base $\boldsymbol{x}$ and the Prime-Irreducible Affinity}
\author[Fusun Akman]{\tiny{Fusun Akman} \\ \\
	Department of Mathematics \\ Illinois State University\\ 
	Normal, IL 61790-4520, USA}
\thanks{\url{akmanf@ilstu.edu}}
 
 \date{}

\maketitle

\begin{abstract}
Arthur Cohn's irreducibility criterion for polynomials with integer coefficients and its generalization connect primes to irreducibles, and integral bases to the variable $x$. As we follow this link, we find that these polynomials are ready to spill two of their secrets: (i) There exists a unique ``base-$x$'' representation of such polynomials that makes the ring $\mathbb{Z}[x]$ into an ordered domain; and (ii) There is a 1-1 correspondence between positive rational primes $p$ and certain infinite sets of irreducible polynomials $f(x)$ that attain the value $p$ at sufficiently large $x$, each generated in finitely many steps from the $p$th cyclotomic polynomial. The base-$x$ representation provides practical conversion methods among numeric bases (not to mention a polynomial factorization algorithm), while the prime-irreducible correspondence puts a new angle on the Bouniakowsky Conjecture, a generalization of Dirichlet's Theorem on Primes in Arithmetic Progressions.
\end{abstract}
                                                                                                                    
\noindent  

\bigskip\noindent\keywords{\small{\bf Mathematics Subject Classification:}  
	{11A41; 11A63; 11N32; 11Y05}

%==========================================================

\section{Introduction.}

%==========================================================

The fact that {\Romanbar{X}} is the Roman numeral for 10, a common base, whereas $x$ is a routine variable for polynomial expressions, is hardly a symbolic conspiracy worthy of the time of Robert Langdon \cite{Br}. Neither would Dr.\ Langdon  spare a smirk for the timeworn joke ``let $x$ be a number, and 10 be a letter.'' However...

%==========================================================

\subsection{Two analogies and a conjecture.}

%==========================================================

 Let  $\Z$ and $\N$ denote the sets of  integers and positive  integers respectively. Polynomials in one variable, $x$,  with integer coefficients, and representations of positive integers in a given base, $b$, are reasonably analogous. Also, we not only think of prime numbers and irreducible polynomials in $\Z[x]$ as   playing similar roles in their respective domains, but expect irreducible polynomials to attain some prime values as well. 
A concrete connection between these analogies is given by Arthur Cohn's Irreducibility Criterion (CIC)~\cite{PS}, where the decimal representation of a positive rational  \textit{prime} $p$ is converted into an \textit{irreducible polynomial} by replacing the base 10 with $x$, which, in turn, takes the value $p$ at $x=10$:
\begin{theorem}[CIC]
	If a prime $p$ is expressed as $p=[a_n\cdots a_1a_0]_{10}$ in base 10, where the integers $a_i$ are digits between 0 and 9, then 	the polynomial $f(x)=a_nx^n+\cdots +a_1x+a_0$ is irreducible in $\Z[x]$.
\end{theorem}
For example, since $p=1187$ is prime, the polynomial 
$ f(x)=x^3+x^2+8x+7\in\Z[x]
$ 
is irreducible. Note that the set of irreducibles in $\Z[x]$ is the union of the rational primes   and the non-constant  polynomials that cannot be written as a product of two polynomials of smaller degree. It is a mystery why this amazingly simple result is not as well known as, say, the Eisenstein criterion for irreducibility.  
%Obviously, CIC is a result that deserves to be displayed in every elementary undergraduate algebra textbook right along with the Rational Roots Theorem and Eisenstein's Criterion. So its proof (though very straightforward) requires some complex analysis. But then, we would want to see a show of hands for who regularly teaches the proofs of the Fundamental Theorem of Algebra, Fermat's Last Theorem, Dirichlet's Theorem on Primes in Arithmetic progressions, and the like! 
%Hence, if $p=2,3,5$, or 7, then the statement holds. The case for two-digit primes is obvious due to the irreducibility of linear polynomials in $\Q[x]$ and the fact that any prime must have relatively prime digits.

Cohn's Irreducibility Criterion was generalized to any base $b\geq 2$ by Brillhart, Filaseta, and Odlyzko~\cite{Co} (also see Ram Murty's  exposition and simpler proof~\cite{Mu}); we will call this result ``Generalized Cohn's Irreducibility Criterion,'' or, GCIC.
\begin{theorem}[GCIC] \label{gcic}
	If a prime $p$ is expressed as $p=[a_n\cdots a_1a_0]_{b}$ in some base $b\geq 2$, where the integers $a_i$ are between 0 and $b-1$, then 	the polynomial $f(x)=a_nx^n+\cdots +a_1x+a_0$ is irreducible in $\Z[x]$.		
\end{theorem}

Does any polynomial that takes a prime value have to be irreducible?  Clearly, this is too much to ask. However, it is true that any polynomial that takes distinct prime values infinitely many times has to be irreducible, since one of its factors in any given factorization in $\Z[x]$ must in turn take one of the values  $\pm1$ infinitely many times, and hence, be a constant. Some irreducible polynomials in $\Z[x]$ cannot attain infinitely many prime values: for instance, the prime 2 divides all values of the irreducible polynomial $f(x)=x^2-x+4$ at $x\in\Z$. 

 Let $f(x)\in\Z[x]$ be a non-constant polynomial. We will call $f$ {\em proper} if   its values at integers are relatively prime, and {\em improper} otherwise. Note that constant polynomials are in neither class by definition. We will also use the common expressions {\em content} and {\em primitive polynomial} to mean the greatest common divisor of all coefficients of a polynomial in $\Z[x]$ and a polynomial with content 1 respectively.

An improper polynomial $f(x)$ with all values divisible by a prime $p$ must have a homomorphic image   $\bar{f}(x)\in\Z_p[x]$ that is divisible by $x^p-x$, as all elements of the field $\Z_p$ are roots of $\bar{f}$. It follows that $p\leq \deg \bar{f}\leq \deg{f}$; hence, the number of such primes is limited by the degree of $f$. Moreover, Dirichlet's Theorem on Primes in Arithmetic Progressions asserts that every primitive {\em linear} polynomial in $\Z[x]$ with positive coefficients takes prime values infinitely many times. The Russian mathematician Viktor Bouniakowsky~\cite{Bo}  made the following conjecture for   irreducible polynomials of all degrees in 1857, generalizing Dirichlet's Theorem. We will refer to it as the ``Bouniakowsky Hypothesis,''  or (BH),    from now on. This conjecture, widely believed to be true by those in the know,  is also conspicuously absent from standard algebra textbooks.
\begin{conjecture}[BH: Bouniakowsky Hypothesis] Let $f(x)\in\Z[x]$ be a proper irreducible polynomial with a positive leading coefficient. Then $f(n)$ takes prime values infinitely many times for $n\in\N$. \label{bh}
\end{conjecture}
As  Ribenboim~\cite{Ri} points out, (BH) is equivalent to the following statement:
\textit{Let $f(x)\in\Z[x]$ be a proper irreducible polynomial with a positive leading coefficient. Then $f(n)$ takes a prime value for some $n\in\N$.} We shall give yet another equivalent statement to (BH) in Theorem~\ref{BHequal}. To date, the only family of polynomials   for which the Bouniakowsky Hypothesis is known to be true is the set of linear primitive polynomials in Dirichlet's Theorem. In fact, not even one proper irreducible polynomial of degree at least two has been shown to have infinitely many prime values. The case   $x^2+1$ is  one of ``Landau's four  problems,''  which were  mentioned  by Edmund Landau at the 1912 International Congress of Mathematicians as ``unattackable at the present state of science''~\cite{Pi}. 

The GCIC construction necessarily gives rise to proper irreducible polynomials:
\begin{proposition} \label{gcicprop}
	Let $f(x)\in\Z[x]$ be a non-constant irreducible polynomial obtained from some base-$b$ expansion of a prime $p$ as described in Theorem~\ref{gcic}. Then $f(x)$ is proper.	
\end{proposition}

\begin{proof}
	If $f(x)=\sum a_ix^i$ with $a_i\geq 0$, then $f'(x)>0$ for all $x\in\N $, and $f$ is strictly increasing for $x\in[1,\infty)$. Moreover, $f$ has no zeros in $\N$. 	Let $b\geq 2$ be the base in which $p$ is expressed. Then $0<f(b-1)<f(b)=p$, and $p$ --the only candidate for a common prime divisor-- cannot divide all values of $f(x)$ at integers.	
\end{proof}

\begin{definition} Let $f(x)\in\Z[x]$.   	If $f(x)$ has positive values for $x\geq k$ for some $k\in\Z$, that is, if it has a positive leading coefficient, then  we will call $f$ {\em positive}, and write $f>0$. A positive, proper, and irreducible polynomial  will be called a PPI-polynomial for short. These polynomials are non-constant by definition of ``proper.''
\end{definition}

Hence, a  nonzero polynomial with nonnegative integer coefficients, such as a GCIC polynomial, is  positive.  Other examples of  positive polynomials include the  \textit{cyclotomic polynomials} $\Phi_k(x)$ for $k\geq 1$. With this background, and after defining base-$x$ representations of polynomials, we will attach an easily-computable infinite family $\cF_p$ that is partly made up of PPI-polynomials, including the ones arising from the GCIC construction, to each rational prime $p$. The defining property of the family will be that each member takes the value $p$ at a sufficiently large value of $x$. It will then be a simple matter to show that this correspondence is 1-1.

%==========================================================

\subsection{A case for a base.}

%==========================================================

The GCIC construction is an example of a broader phenomenon: it is a well-known trick that a non-constant polynomial with nonnegative integer coefficients can be reconstructed from only two of its values, $f(c)=b$ and $f(f(c))=f(b)$ (as long as $b\geq 2$),  for some positive integer $c$. For example, suppose that we want to recover the polynomial 
$f(x)=3x^2+5x+7
$ 
from its values $f(1)=15$ and $f(15)=757  $. We write $757$ in base $15$, and use the digits as coefficients for the polynomial:
\[ 757=[ 357  ]_{15}\Longrightarrow f(x)=3x^2+5x+7.
\]
The reason why we can do this is   that the given polynomial provides a {\em model} for the  representation of any positive integer in base $b>\max\{3,5,7\}$. Since we are not supposed to know the coefficients, the  value $b=f(1)$, which is the sum of the positive coefficients, or else the values $f(2),f(3),\dots$, which are even larger, will satisfy this condition. Let us go even further. What if we have a positive polynomial with some  negative coefficients? 
\begin{example} Consider the  positive polynomial 
	\bea f(x)%&=&\Phi_{12}(x)+\Phi_{26}(x)\\
	&=&2-x-x^3+2x^4-x^5+x^6-x^7+x^8-x^9+x^{10}-x^{11}+x^{12}.
	\eea
	We will soon show that any positive polynomial  has a unique representation in ``base $x$'' with constant and linear ``digits'' that will work for any sufficiently large base $b$. Indeed, when we write our particular polynomial in the form
	\bea f(x)&=& 2+(x-1)x+(x-1)x^2+(x-2)x^3 +x^4+(x-1)x^5 +(x-1)x^7\\ &&+(x-1)x^9
	+(x-1)x^{11}\\
	&=& [ (x-1)(0)(x-1)(0)(x-1)(0)(x-1)(1)(x-2)(x-1)(x-1)(2) ]_x,
	\eea
	 we find that
	\[ f(3)=[202020211222]_3,\;\; f(4)=[303030312332]_4,\;\; f(5)=[404040413442]_5.
	\]
	While some digits are repeated as is, others are always one less than or two less than the base, exactly as indicated by the ``coefficients'' of the $x^i$. 
	However, since $2$ is a coefficient itself, it is too small to serve as a base, and  we do not expect $ f(2)=[101010111000]_2  
	$ to conform to this shape.
	
\end{example}
These numerical experiments then suggest that a base-$x$ representation of a    positive polynomial can be obtained simply by writing two  values $f(b_1)$ and $f(b_2)$ in bases $b_1,b_2\geq 2$ respectively, where the integers $b_i$ are  sufficiently greater than the largest absolute value of the coefficients of $f$ (see Theorem~\ref{factoralg} for an application of this idea to factorization of polynomials, and the proof of Theorem~\ref{basex} where we use it establish uniqueness). However, this is hardly necessary in practice. In the next section, we will  describe how to compute the base-$x$ representation of any  positive polynomial algorithmically, using countably many  digits that form a chain with minimum and maximum elements, and   show that $\Z[x]$ naturally becomes an ordered domain. This representation arguably provides the most compelling analogy between $\Z[x]$ and $\Z$. 

%==========================================================

\section{Representation in base $x$.}

%==========================================================

\subsection{Existence and uniqueness.}\label{exun}

%==========================================================

The example above  prompts us to give the following definitions:
 Let $f(x)$ be a positive polynomial. 
	We will call the nonnegative integer coefficients $a_i$ the	{\em constant digits} and the linear coefficients of the form $(x-a_i)$ the {\em linear digits} in a base-$x$ representation of $f(x)$.  These infinitely many digits form our   {\em polynomial alphabet}.

\begin{theorem}[Base-$x$ representation] \label{basex}
	Let $f(x)\in\Z[x]$ be a      positive polynomial. Then there exists a unique least positive integer $a$  and a unique	{\em base-$x$} representation of $f(x)$  as follows:
	\[ f(x)=b_0(x)\, 1+b_1(x)\, x+b_2(x)\, x^2+\cdots +b_m(x)\, x^m,
	\]
	with 	
	\[ b_i(x)\in\{ 0,1,\dots,a-1\}\cup\{ x-1,\dots,x-a\},\;\;\; b_m(x)\neq 0 .\]
	In particular,  the unique representation of   $f(b)\in\N$ in base $b$ for all $b\geq  a$ is given by
	\[ f(b)=[b_m(b) \cdots b_1(b) \, b_0(b) ]_{b}.
	\]
\end{theorem}

\begin{proof}
	We start the construction with the constant term and work our way towards higher powers of $x$. At any  step, if $ a_i$ is a negative coefficient of $x^i$, then we replace  it by the linear digit $(x-|a_i|)$, and add the term $-x^{i+1}$ to balance it out. If the original coefficient of $x^{i+1}$ in $f$ is positive, then it will be reduced by one and will not affect the coefficient of $x^{i+2}$. However, if $ a_{i+1}\leq 0$, then the new ``coefficient'' of $x^{i+1}$ will be the linear digit $(x-|a_{i+1}|-1)$, and will cause a new term $-x^{i+2}$ to appear, etc. We can see that positive coefficients remain the same or are reduced by one, whereas  negative coefficients $ a_i $ are eventually replaced by $(x-|a_i|)$ or $(x-|a_{i+1}|-1)$. This rule also applies to the last (positive)  coefficient  of $f(x)$, and the   original highest power will now have a smaller positive coefficient or else vanish. The base-$b$ representations of $f(b)$ in $\Z$ follow for $b\geq a$.  
	%The uniqueness of such a polynomial base-$x$ representation can be shown as follows. 
	Now suppose that there are two distinct base-$x$ representations of the same polynomial. Pick any $b>  \max\{ a_i+b_i\}$ in $\N$ so that the   nonnegative integers $a_i$ and $b_i$ that appear in the $i$th constant or linear digits  of the two forms  of the polynomial $f$ always result in  a pair of distinct values $\{ a_i,b-b_i\}$ and a pair of distinct values $\{ b_i,b-a_i\}$.  Then the positive integer $f(b)$ must have two distinct base-$b$ representations, as we shall see: let the $i$th digits be distinct in the two base-$x$ expressions for some fixed $i$. They cannot both be constant digits or both be linear digits, because these would immediately give us different  integers in base $b$ upon substitution of $b$. Due to the choice of $b$, the two values will be different even in the case of a pair of opposite-type $i$th digits.
\end{proof}

	Given any positive polynomial $f(x)=\sum a_ix^i\in\Z[x]$, we will call the unique least positive integer $a$ described in Theorem~\ref{basex}  the {\em minimum base} of $f$, and denote it by $a=\mb(f)$. 
	We will also define  $H(f)=\max |a_i|$ as the {\em height} of $f$. By construction, it is clear that
$ H(f)\leq \mb(f)\leq H(f)+1$.

\begin{example}
	We compute the base-$x$ representation of
	\bea f(x)&=&   -7+2x^2-x^3+x^4\\
	&=&(x-7)  1-x+2x^2-x^3+x^4\\
	&=& (x-7)  1+(x-1) x-x^2+2x^2-x^3+x^4\\
	&=&(x-7)  1+(x-1) x + x^2-x^3+x^4\\
	&=& (x-7)  1+(x-1) x + x^2+(x-1)x^3-x^4+x^4\\
	&=& (x-7)  1+(x-1) x + x^2+(x-1)x^3\\
	&=&[(x-1)(1)(x-1)(x-7)]_x.
	\eea  We find that the minimum base of $f$ is $\mb(f)=7=H(f)$. By contrast, the polynomial $g(x)=  7+2x^2+x^3+x^4$ has minimum base $\mb(g)=8=H(g)+1$. 
\end{example}

%==========================================================

 \subsection{A Natural Linear Ordering on  $\boldsymbol{\Z[x]}$.}\label{order}
 
 %==========================================================
 
The following statement is indisputable, yet somewhat surreal:
\begin{theorem} [Linear Ordering] The linear ordering
	\[ (0)< (1)< \cdots <(n)<\cdots  < (x-n) <\cdots < (x-2) <(x-1) 
	\]
	on the polynomial alphabet extends to a natural linear ordering on the positive polynomials and to one on $\Z[x]$ by symmetry: we have
	\bea  &&[(0)]_x\\
	&<& [(1)]_x <\cdots    <[(x-1)]_x <[(1)(0)]_x\\
	&<& [(1)(1)]_x<\cdots <[(1)(x-1)]_x<[(2)(0)]_x\\
	&\vdots& \\
	&<& [(x-2)(1)]_x<\cdots <[(x-2)(x-1)]_x <[(x-1)(0)]_x  \\
	&<& [(x-1)(1)]_x<\cdots <[(x-1)(x-1)]_x <[(1)(0)(0)]_x \\
	&<& [(1)(0)(1)]_x          <\cdots .
	\eea
	Every polynomial in $\Z[x]$ has a successor and a predecessor   obtained by adding or subtracting $[(1)]_x$ respectively.
\end{theorem}

Note that the largest digit plus 1 is equal  to the base, $x$, written $[(1)(0)]_x$, and that the subset of positive elements of $\Z[x]$ is not well-ordered. If $m>n>0$, then a positive polynomial with $m$ digits is strictly greater than one with $n$ digits. If two positive polynomials  have the same number of digits, then the lexicographic order determines which one is larger.

\begin{corollary} \label{rules}
	The ordering on $\Z[x]$ makes it  an ordered integral domain. In particular,  $f>g$ if and only if $f-g$ is  positive in the sense that we have defined. If $f>g$, then the following properties hold:
	\[ (i)\; f+h>g+h\;\;\mbox{for all $h$, and}\;\;(ii)\;  fh>gh\;\;\mbox{for all $h>0$}.
	\]
\end{corollary}

	The field of fractions  of $\Z[x]$, and  its subring $\Q[x]$, are also ordered according to the rule
	\[ \frac{f}{g}>0\Longleftrightarrow \mbox{$f$ and $g$ are both positive or both negative.}
	\] 
	In fact, the set of rational functions in $x$ over any ordered field $\F$ is an ordered field itself, whose order relation is compatible  with the one on $\,\F$. However, this fact is more commonly used in   analysis as an example of a non-Archimedean field: no natural number $n$ satisfies $n\cdot 1>x$ (e.g., see Example~3.2.6 and Exercise 3.16 of Lay~\cite{Lay}). In algebra and number theory, on the other hand, we can make use of this ordering to establish a more natural version of the Division Algorithm, as we shall see below.

\begin{theorem}[Division Algorithm  for a  Monic Divisor]\label{divalg}
	Let $f(x)$ be any polynomial and $g(x)$ be a monic polynomial in $\Z[x]$. Then there exist    unique   polynomials $q(x),r(x)\in \Z[x]$ such that
	\[ f(x)=q(x)g(x)+r(x),\;\;\; 0\leq r(x)< g(x).
	\]
\end{theorem}

\begin{proof} 
	Existence: Long division gives us the correct form  $f=qg+r$, but the second condition is replaced by  ``$r=0$ or $\deg r<\deg g$''. If $r\geq 0$, then we are done, as a larger degree means a larger polynomial when both are positive. However, if $r<0$, then $q(x)-1$ as the quotient and $g(x)+r(x)$ as the remainder will yield the desired result  by judicious applications of Corollary~\ref{rules}. In the latter case, the degree of the remainder will be equal to the degree of $g$. 
	
	Uniqueness: Suppose that we have 
	$ f(x)=q_1(x)g(x)+r_1(x)$ and $f(x)=q_2(x)g(x)+r_2(x)$, 	with $0\leq r_1(x),r_2(x)< g(x)	$. Then 
	$  [q_1(x)-q_2(x)]g(x)=r_2(x)-r_1(x) 
	$. 
	If both sides are zero, then we are done. Suppose not. Thus, $\deg (r_2-r_1)\geq \deg g$ by the additivity of degrees in polynomial multiplication over an integral domain. However,  we also know that $\deg (r_2-r_1)\leq \deg g$ (neither remainder can have a degree larger than that of $g$), and hence, we must have $\deg (r_2-r_1)= \deg g$. This leads us to $\deg (q_1-q_2)=0$. Without loss of generality, say $q_1(x)=q_2(x)+a$ for some $a\in\N$. Hence, we conclude that \[ a \, g(x)=r_2(x)-r_1(x)\Longrightarrow r_2(x)=r_1(x)+a\, g(x)\geq 0+1\, g(x)=g(x) 
	\]
	by Corollary~\ref{rules},	a contradiction.
	
\end{proof}

Arithmetic operations on positive polynomials written in base $x$, including digital division by a monic divisor, are covered in  Appendix~\ref{app} (see the last example). Unfortunately, we cannot extend the statement and the proof to any polynomial ring $\F[x]$ where $\F$ is a generic ordered domain or field, because not every ordered domain is {\em discrete}: there may be positive elements between 0 and 1.

%==========================================================

\subsection{Polynomial factorization via base patterns.}\label{proof}

%==========================================================

Factorization of integers into prime factors and of polynomials in $\Z[x]$ into irreducible factors in finitely many steps is theoretically possible. The reader can consult Knuth~(\cite{Kn}, pp.~420-436) for a compendium of polynomial factorization algorithms  and their history. Our method of factorization of positive polynomials is a variant of \textit{Kronecker's method}: according to Knuth (p.~431), Isaac Newton had described a method to factorize polynomials in $\Z[x]$ in {\em Arithmetica Universalis} (1707). Newton's method was then generalized into a general, finite algorithm by the astronomer  Friedrich von Schubert in 1793 (see the account of the mathematics historian M.\ Cantor~\cite{Cr}, pp.~136-137), and Kronecker rediscovered the method independently about 90 years later~\cite{Kr}. The complete description of Kronecker's algorithm can be found in van der Waerden~\cite{vdW}, pp.~97-98.  The algorithm relies on the facts that (1) if $g(x)$ divides $f(x)$, then for all $a\in\Z$, $g(a)$ either  is zero or  divides $f(a)$; (2) knowing sufficiently many values of a polynomial of bounded degree, we can reconstruct it using Lagrange's or Newton's interpolation formulas; (3) hence, there are finitely many candidates for polynomial factors; and 
(4) the division algorithm for rational numbers can be applied to see whether a suspected factor is an actual factor.  This algorithm is too costly for computer applications and needs to be enhanced or replaced by other techniques such as looking for a rational root first, or reduction modulo a few primes. Similarly, the procedure that we are going to describe is computationally expensive, but fun to implement by hand. A slight improvement over Kronecker's method is that only two values of the polynomial are required, so that fewer cases need  to be considered. Moreover, interpolation formulas are replaced by base-$x$ representations.

 We recall that for $f(x)=\sum a_i x^i\in\C[x]$, the {\em $\ell_2$-norm} of $f$ is defined as
\[ ||f||=\left( \sum |a_i|^2   \right)^{1/2}.
\]
We will also use the traditional divisibility notation,  $g |f$, for both polynomials and integers.
 If $f(x)\in\Z[x]$ is positive, then it has finitely many positive divisors $g(x)\in\Z[x]$, and we define the {\em minimum factor base} of $f$ to be
	\[ \mfb(f)=\max\{ \mb(g): g|f,\; g>0   \}   .
	\]
	Note that $\mfb(f)=\mb(f)$ when $f$ is irreducible. Although the exact value of $\mfb(f)$ is obviously unknown in a factorization problem, upper bounds for the height of any divisor entirely in terms of $f$ are known. For example, a corollary of Theorem~4 in Mignotte~\cite{Mi} (namely, that $||g||\leq 2^{\deg g}||f||$) implies the following  bound for $g$ dividing $f$:
\[ H(g)\leq ||g||\leq 2^{\deg g}||f|| \leq 2^{\deg f}||f||.
\]
We will not be concerned with the best known upper bound for heights of divisors, but make a note of the fact that they do exist and depend entirely on $f$ itself. This, in turn, provides us with an upper bound on the minimal bases of all positive divisors of $f$. In what follows, let $ \overline{\mbox{MFB}} (f)$ denote such an absolute upper bound for $\mfb(f)$, which is not unique.

\begin{theorem}\label{factoralg} It is possible to factor a positive polynomial $f(x)\in\Z[x]$  into irreducible polynomials in $\Z[x]$ in finitely many steps by factorizing two values
$f(b_1)$ and $f(b_2)$ 	of $f$ into primes for distinct natural numbers $b_i>  \overline{\mbox{MFB}} (f)+1$.		
\end{theorem}

\begin{proof} 
	Given a positive polynomial $f(x)$, we factor out its content first.  Assume that $f$ is primitive. 
	Let $f(x)=g(x)h(x)$ be an unspecified  factorization into positive and necessarily primitive  polynomials in $\Z[x]$.  Let $a=\overline{\mbox{MFB}} (f) $, so that all integers of the form $f(b)$ as well as their unknown factors $g(b)$ and $h(b)$ for $b\geq a$ have unique representations in base $b$ that match the base-$x$ representations of the corresponding polynomials.  We first  factor $f(a+2)$ and $f(a+3)$ into primes and compute all   positive  divisors of these two numbers. By comparing all possible divisor pairs of $f(a+2)$ and $f(a+3)$, we can uncover a common base pattern that is repeated as in the last example of the Introduction (Section~1). Any observed base  pattern can be converted into a positive polynomial and be verified or ruled out as a divisor. We can then start testing the quotients for divisors in turn. 
	
	We remark that (1) two substitutions are necessary to decide whether a digit in the given base (substituted value)  corresponds to a fixed or variable digit in base $x$, and (2) it is possible to choose even larger numbers $b$ such that $f(b)$ has fewer factors. However, the latter course requires additional computations, which may be undesirable. 
	
	Why can't we just use $\overline{\mbox{MFB}} (f)$ as the least possible value for $b$, instead of $\overline{\mbox{MFB}} (f)+2$? The only thing that could potentially go wrong in the first paragraph's scenario is that there may be additional non-constant polynomial factors of $f(x)$ that attain the value $+1$ at $b_1$ and $b_2$; however, if a divisor $g(x)$ is not constant, then it has at least one term of the form $ux^j$ ($u,j\geq 1$) or $(x-u)x^j$ ($u\geq 1$, $j\geq 0$)   in its base-$x$ representation, making $\mfb(f)\geq \mb(g)\geq u+1$ or $u$ respectively. Therefore, $g(x)$ cannot attain the value 1 when evaluated at an $x=b_i\geq \overline{\mbox{MFB}} (f)+2\geq \mfb(f)+2\geq u+2$, because in this case at least one of its  terms in the base-$x$ representation ($ux^j$ or $(x-u)x^j$)  will achieve a value $\geq 2$, 
	and the possibility of an invisible non-constant factor is avoided. 
\end{proof}

\begin{example} Let us factor a polynomial given in an exercise of van der Waerden~\cite{vdW}: $f(x)=x^5+x^4+x^2+x+2$. We compute 
	\[ ||f||\approx 2.828\;\;\mbox{and}\;\; H(g)\leq 2^5||f||\approx 90.496\;\;\mbox{for any $g|f$.} 
	\]
Hence, we may take	$\overline{\mbox{MFB}} (f)=90+1=91$, and evaluate
\[ f(93)=7\, 031\, 697\, 638=2^1\cdot 7^1\cdot 1\, 249^1\cdot 402\, 133^1 
\]
and
\[ f(94)=7\, 417\, 124\, 052=2^2\cdot 3^1\cdot 13^2\cdot 229^1\cdot 15\, 971^1.
\]	
By  trial and error, we find that the factors $8\, 743=7\cdot 1\, 249$ and 	$8\, 931=3\cdot 13\cdot 229$ of $f(93)$ and $f(94)$ respectively have the base representations
\[ 8\, 743=93^2+93+1=[111]_{93}\;\;\mbox{and}\;\; 8\, 931=94^2+94+1=[111]_{94}.
\]
That is, $g(x)=x^2+x+1$ is a plausible factor, which is  irreducible. We divide $f(x)$ by $g(x)$ to find the exact quotient $h(x)=x^3-x+2$. By the rational root theorem, $h(x)$ is irreducible as well, and $f(x)=g(x)h(x)$ is the full factorization of $f(x)$ into irreducibles.	

Now, if we had tried the factors 
\[ 804\, 266=2\cdot 402\, 133=92\cdot 93^2+92\cdot 93+2=[(93-1)(93-1)(2)]_{93}
\]	
	and
	\[ 830\, 492=2^2\cdot 13\cdot 15\, 971=93\cdot 94^2+93\cdot 94+2=[(94-1)(94-1)(2)]_{94}
	\]
of $f(93)$ and 	$f(94)$ respectively, we would have decided that
\[ [(x-1)(x-1)(2)]_x=(x-1)x^2+(x-1)x+2=x^3-x+2
\]
might be a factor, and we would have been correct.

\end{example}

%\begin{example} Let us consider the primitive polynomial
%	\[f(x)=(x+1)(x^3-2)=x^4+x^3-2x-2 =[(1)(0)(x-1)(x-3)(x-2)    ]_x 
%	\]
%	with   	$ x^3-2=[(x-1)(x-1)(x-2)]_x$.    	We  secretly know that $\mfb(f)=3$, so that choosing $b_i\geq 5$ will be sufficient. We compute
%	\[  f(5)=738=2^1\cdot 3^2\cdot 41^1
%	\;\;\mbox{and}\;\; f(6)=1498= 2^1\cdot 7^1\cdot 107^1  .
%	\]
%	First, we see that 2 divides both values, and that
%	$ \bar{f}(x)=x^3(x+1)\in\Z_2[x]
%	$
%	is divisible by $x^2-x=x(x+1)$, showing us that $f$ is improper. Next, we  notice that
%	\[ 6=2\cdot 3=[11]_5 \;\;\mbox{and}\;\; 7=[11]_6         .
%	\]
%	This tells us that the irreducible polynomial $x+1$ could be a divisor, with some quotient $q(x)$.  Eliminating the corresponding factors from the prime factorizations, we look at
%	\[ q(5)=3\cdot 41=123=[443]_5
%	\;\;\mbox{and}\;\;    q(6)=2\cdot 107=214=[554 ]_6.       
%	\]	
%	The pattern suggests that 
%	\[ q(x)=[(x-1)(x-1)(x-2)]_x=x^3-2
%	\]
%	is a   factor of $f$ (which we could have verified when we found $x+1$). Both factors check out. The only question is whether $x^3-2$ is irreducible over $\Z$, so we try finding common patterns among the remaining pairs of factors of $q(5)$ and $q(6)$:  
%	\[  3=[3]_5,\; 41=  [131]_5,\;              2=[2]_6,\; 107=[255]_6.
%	\]
%	No new pattern emerges, and we are done. (Note that $x$ is not a digit.)
%\end{example}

Every polynomial factorization algorithm serves as an irreducibility test as well.  
The following corollary is a generalization of GCIC, where all coefficients are not necessarily positive.

\begin{corollary} \label{irredtest}
	Let $f(x)$ be a non-constant positive polynomial and $a= \overline{\mbox{MFB}}(f)$. If $f(b)=p$, a prime, for some integer $b>a+1$, then $f(x)$ is proper and irreducible in $\Z[x]$.  	
\end{corollary}
\begin{proof}
	The polynomial $f$ must be proper since it is increasing for $x\geq a$, and $f(a)$ cannot be divisible by $p$.
\end{proof}

\begin{remark} It has been pointed out to us  that the bounds in Theorem~\ref{factoralg} and Corollary~\ref{irredtest} can be improved in several ways, for example, by using ideas from Murty's paper~\cite{Mu}. Murty proves the following: if  $f(x)=a_mx^m+\cdots +a_0$ and  $f(b)$ is prime for some $b\geq \max_{0\leq i\leq m-1}|a_i/a_m|+2$, then $f(x)$ is irreducible (which falls just short of proving Cohn's Irreducibility Criterion). The statement and proof of Theorem~\ref{factoralg} can also be fine-tuned. However, we feel that the pedagogical value of the last two statements is in the novelty and simplicity of their proofs that make use of a polynomial base, with no explicit need to consider complex roots, and where the only additional ingredient is Mignotte's inequality (or similar results bounding the height). We do not see our factorization method, and indeed the base conversion algorithms in the next section, as being competitive with established computer algorithms. Small examples are just fun and instructive to compute by hand.	
\end{remark}

%==========================================================

\section{Polynomial and numeric bases.} \label{primevalues}

%==========================================================

\subsection{Polynomial representatives of positive integers in base $b$.}

%==========================================================

Since base-$x$ representations of polynomials are central to our paper, we introduce  special notation and terminology for the polynomials that represent base-$b$ expansions of natural numbers:
	Let $c\in\N$.
	\begin{enumerate}
		
		\item The polynomial
		\[ f_c^{(1)}(x)=\frac{x^c-1}{x-1}=x^{c-1}+\cdots +x+1
		\]	
		will be called the {\em polynomial representative of $c$ in base 1}.
		\item Let $b\in\N$, with $b\geq 2$. If 
		$ c=[a_k\cdots a_1a_0]_{b}
		$ 
		is the unique expansion of $c$ in base $b$, with digits $a_i$ between $0$ and $b-1$,  then the polynomial
		\[ f_c^{(b)}(x)=a_kx^k+\cdots +a_1x+a_0
		\]
		will be called the {\em polynomial representative of $c$ in base $b$}.
		
	\end{enumerate}	
	In both of these cases, we have $f_c^{(b)}(b)=c$.

Table~1 below exhibits all   polynomial representatives $f_{17}^{(b)}(x)$ of the prime 17 in bases $b\geq 1$. The first one is irreducible since it is the 17th cyclotomic polynomial, and the rest are irreducible as a result of GCIC. 
 We will show how to generate this table for any $c\in\N$ in a systematic way starting from $f_c^{(1)}(x)$ in the next Proposition. 
	\[\begin{array}{|c|c||c|c||c|c|} \hline  \mbox{Base $b$} & f_{17}^{(b)}(x)   &   \mbox{Base $b$} & f_{17}^{(b)}(x) &  \mbox{Base $b$} & f_{17}^{(b)}(x)         \\  \hline  
1 & x^{16}+\cdots +x+1=\Phi_{17}(x) & 8 & 2x+1 & 15 & x+2                 \\
2 & x^4+1=\Phi_8(x)  &  9 & x+8 &       16 & x+1=\Phi_2(x) \\              
3 & x^2+2x+2=\Phi_4(x+1) &  10 & x+7 &   17 & x         \\ 
4 & x^2+1=\Phi_4(x) &    11 & x+6 &   18 &  17        \\ 
5 & 3x+2   &    12 & x+5 &    19 &    17              \\ 
6 & 2x+5 & 13 & x+4 &   20 & 17                 \\ 
7 & 2x+3 &    14 & x+3 &    \cdots & \cdots  \\   \hline
\end{array}
\] 

\begin{center}  {\small  {\bf Table 1.}  Polynomial representatives of $p=17$           }                          \end{center}	
	%\caption{Polynomial representatives of $p=17$}\label{p17}
%\end{table}

%==========================================================

\subsection{Base conversions.}

%==========================================================

\begin{proposition}[Base conversions]\label{seed} Let $a,b,c\in\N$,   $f(x)\in\Z[x]$ be  either a nonzero polynomial with nonnegative coefficients $\leq b-1$ where $b\geq 2$ and $f(b)=c$,   or   $f(x)=(x^c-1)/(x-1)$ (in which case we set $b=1$). Then $f(x)=f_c^{(b)}(x)$, and the following hold:  
	\begin{enumerate}
		
		\item If $c=p$, a prime, then   $f_{c}^{(b)}$ is a PPI-polynomial for all $b\geq 1$.
		\item (Descent) To convert $f(x)=f_c^{(b)}(x)$ into $f_c^{(b-a)}(x)$, where $b-a\in\N$ as well, follow these steps:
		\begin{enumerate}
			\item Substitute $x+a$ for $x$ and expand $f(x+a)$ in powers of $x$. This polynomial will have all positive coefficients. The indeterminate $x$ now represents the base $b-a$.
			\item If all coefficients of the new polynomial are strictly less than $b-a$, stop. Otherwise, proceed to the next step. (If $b-a=1$, then the stopping point is reached when all coefficients are 1.)
			\item Replace all coefficients by their base-$(b-a)$  expansions; substitute $x$ for each $b-a$; simplify. Go to Step (b).  
		\end{enumerate}
		\item (Ascent)	To convert $f(x)=f_c^{(b)}(x)$ into $f_c^{(b+a)}(x)$, follow these steps:
		\begin{enumerate}
			\item Substitute $x-a$ for $x$ and simplify $f(x-a)$. The indeterminate $x$ now represents the base $b+a$.
			\item Convert the positive part of each coefficient to base $b+a$; substitute $x$ for $b+a$; simplify.
			\item If the absolute value of each coefficient in the new polynomial is strictly less than $b+a$, go on to the next step. Otherwise, return to Step (b).
			\item Write the polynomial in base $x$; replace  any digit of the form $x-j$ by $b+a-j$; simplify.
		\end{enumerate}
		\item The ascent and descent procedures are inverses of each other.	
		\item If $1\leq b<b'$ and $b\leq c$, then $f_c^{(b)}>f_c^{(b')}$. For $b>c$, both polynomial representatives are equal to the constant $c$. 
	\end{enumerate}
	
\end{proposition} 

\begin{proof}
	Part~(1) follows from  the properties of cyclotomic polynomials and Proposition~\ref{gcicprop}. In Parts (2) and (3),   we are simply performing naive base conversions by starting with an approximate expansion and making corrections as necessary: the unique expression is independent of the method. The only point that may not be immediately obvious is conversion from base 2 to base 1, or, to the unary system. Why do we always obtain the lowest $c$ nonnegative powers of $x$ and not just any $c$ distinct powers? This is proven in the following Lemma. Part~(4) is clear by uniqueness of base expansions. Finally, Part~(5) can be seen in the details of the descent procedure, where $x$ is replaced by $x+a$ and $b-a$ is replaced by $x$; all other terms and factors are nonnegative, and we can apply properties of ordered domains.
\end{proof}

In computing, the convention for the unary representation of a positive integer $c$ is $c$ consecutive digits of 1, with no zeros in between or outside. Here is a good reason why:

\begin{lemma} \label{unary}
	Let $c\in\N$, and 
	\[ f(x)=x^k+a_{k-1}x^{k-1}+\cdots +a_1x+a_0\in \Z[x],
	\]	
	with $a_i\in\{ 0,1\}$. Moreover, let $f(2)=c$. Then $f(x)=f_c^{(2)}(x)$, and the descent process in Proposition~\ref{seed} Part~(2) gives us
	\[ f_c^{(1)}(x)=x^{c-1}+\cdots +x+1.
	\]
\end{lemma}
\begin{proof}
	By induction. After $f(x)$ is replaced by the larger polynomial
	\[ g(x)=f(x+1)=(x+1)^k+a_{k-1}(x+1)^{k-1}+\cdots +a_1(x+1)+a_0,
	\]	
	all coefficients of $g(x)$ up to the leading one become strictly positive, and the constant term becomes	
$ g(0)=f(1)=1+a_{k-1}+\cdots +a_1+a_0
$. 
	As instructed, we replace all coefficients $a$ of $g$ by $x^{a-1}+\cdots +x+1$, and obtain a larger polynomial. If $f(1)$ is not already equal to 1, then it will be replaced by 
	$ [x^{f(1)-1}+\cdots +x]+1 
	$, 
	where the rest of the polynomial $g(x)$ is $O(x)$. Therefore, the constant term permanently becomes 1, and no powers of $x$ in $g$ up to the largest one will ever vanish. Let $a\in\N$ be the coefficient of $x$ in $g$. If $a> 1$, then the next step brings the term $(x^{a-1}+\cdots +x+1)x$ in addition to the constant term 1, and the rest of the polynomial is now $O(x^2)$. In this manner, we guarantee that the smallest two terms of the evolving polynomial will be $x+1$ for the rest of the procedure, and so on, until we reach 
$ x^{c-1}+\cdots +x+1 
$, 
	which gives us the value $c$ at $x=b=1$.	 
\end{proof}

\begin{example} \label{ex17} Let us verify a few examples from Table~1. In ascending order, we have
$17=[10001]_2$, giving us $f_{17}^{(2)}(x)=x^4+1$. The next base is $b=3=2+1$. We compute $  f(x-1)	=x^4-4x^3+6x^2-4x+2$,  then replace $4=b+1$ by $x+1$ and $6=2b\,$ by $2x$, and obtain the polynomial
\[ x^3-x^2-x+2=(x-2)x^2+(x-1)x+2.
\]
A final substitution of $b=3$ into the digits only gives us $f_{17}^{(3)}(x)= x^2+2x+2$. 
%	\bea
%	&& 17=10001_2\mapsto \boxed{f(x)=x^4+1}\;\;\mbox{\rm for $b=2$}\\
%	&& f(x-1)=(x-1)^4+1=x^4-4x^3+6x^2-4x+2 \\
%	&&\mapsto x^4-(x+1)x^3+(2x)x^2-(x+1)x+2\\
%	&& =x^3-x^2-x+2=(x-2)x^2+(x-1)x+2\mapsto \boxed{f(x)= x^2+2x+2  }\;\;\mbox{\rm for $b=3$}.
	%&& f(x-1)=(x-1)^2+2(x-1)+2\mapsto \boxed{f(x)= x^2+1 }\;\;\mbox{\rm for $b=4$}\\
	%&& f(x-1)=(x-1)^2+1=x^2-2x+2= (x-2)x+2\\
	%&& \mapsto  \boxed{f(x)=3x+2}\;\;\mbox{\rm for $b=5$}\\
	%&& f(x-1)=3x-1=2x+(x-1)\mapsto \boxed{f(x)=2x+5}\;\;\mbox{\rm for $b=6$}\\
	%&& f(x-1)= 2x+3\mapsto \boxed{f(x)=2x+3}\;\;\mbox{\rm for $b=7$}\\
	%&& f(x-1)= 2x+1\mapsto \boxed{f(x)=2x+1}\;\;\mbox{\rm for $b=8$}\\
	%&& f(x-1)=x+(x-1)\mapsto \boxed{f(x)=x+8}\;\;\mbox{\rm for $b=9$}\\
	%&& f(x-1)=x+7\mapsto \boxed{f(x)=x+7}\;\;\mbox{\rm for $b=10$}\\
	%&&\vdots
%	\eea
	In descending order, we know that $f_{17}^{(17)}(x)=x$, and $f_{17}^{(16)}= f_{17}^{(17)}(x+1)=x+1$ (no change needed). Let us pick up the sequence at $f_{17}^{(9)}(x)=x+8$. The next base is $b=8=9-1$. Hence,
	\[ f_{17}^{(9)}(x+1)=x+9=x+(8+1)\mapsto x+(x+1)=2x+1=f_{17}^{(8)}(x).
	\]
	%\bea 
	%&& \boxed{f(x)=x} \;\;\mbox{\rm for $b=17$}\\
%	&& f(x+1)=x+1\mapsto \boxed{f(x)=x+1} \;\;\mbox{\rm for $b=16$}\\
%	&& \vdots \\
%	&& \boxed{f(x)=x+8} \;\;\mbox{\rm for $b=9$}\\
%	&& f(x+1)=x+9=x+(8+1)\mapsto  x+(x+1) \mapsto  \boxed{f(x)=2x+1} \;\;\mbox{\rm for $b=8$}.
	%&& f(x+1)=2x+3\mapsto \boxed{f(x)=2x+3} \;\;\mbox{\rm for $b=7$}\\
	%&& \boxed{f(x)=2x+5} \;\;\mbox{\rm for $b=6$}\\
	%&& f(x+1)=2x+7=2x+(5+2)\mapsto 2x+(x+2)\mapsto  \boxed{f(x)=3x+2} \;\;\mbox{\rm for $b=5$}\\
	%&& f(x+1)=3x+(4+1)\mapsto 3x+(x+1)\mapsto 4x+1\mapsto x\cdot x +1\mapsto \\
	%&& \mapsto  \boxed{f(x)=x^2+1} \;\;\mbox{\rm for $b=4$}\\
	%&& f(x+1)=x^2+2x+2\mapsto \boxed{f(x)=x^2+2x+2} \;\;\mbox{\rm for $b=3$}\eea \bea
	%&& f(x+1)=x^2+4x+5=x^2+(2^2)x+(2^2+1)\mapsto x^2+(x^2)x+(x^2+1)\\
	%&& = x^3+2x^2+1\mapsto x^3+x\cdot x^2+1=2x^3+1\mapsto x\cdot x^3+1\\
	%&& \boxed{f(x)=x^4+1} \;\;\mbox{\rm for $b=2$}.
%	\eea
%	The ascent from base 1 to base 2 is long, but it can be verified in finitely many steps:	
%	\bea   \phi_{17}(x-1)&\mapsto& x^{22}-2x^{21}+x^{20}-3x^{18}-x^{17}-2x^{16}-x^{15}+4x^{14}-3x^{13}-2x^{12}\\
%	&&+x^6-x^5-x^4+1\\
%	&\mapsto& x^{22}-x^{22}+x^{20}-(x+1)x^{18}-x^{17}- x^{17}-x^{15}+ x^{16}-(x+1)x^{13}\\
%	&&- x^{13}+x^6-x^5-x^4+1\\
%	&\mapsto& x^{20}-x^{19}-x^{18}-2x^{17}+x^{16}-x^{15}-x^{14}-2x^{13}+x^6\\
%	&&-x^5-x^4+1\\
%	&\mapsto&\cdots \\
%	&\mapsto& x^6-x^5-x^4+1\;\;\mbox{\rm (only coefficients of $0,\pm 1$ remain)}\\
%	&=&1+(x-1)x^4+(x-2)x^5\;\;\mbox{\rm (substitute $x=2$)}\\
%	&\mapsto& 1+x^4.
%	\eea
	
\end{example}

%==========================================================

\section{Families of irreducibles with the same prime value at large integers.}

%==========================================================

The GCIC construction as well as Corollary~\ref{irredtest} make it clear that only the prime values of  polynomials at large $n\in\N$ are significant for characterizing irreducibility. Given a positive rational prime $p$, let us define a family $\cF_p\subset \Z[x]$ that consists of the constant polynomial $p$ and the PPI-polynomials with this property:

\begin{definition} \label{fp}
	Given a positive rational prime $p$, consider the constant polynomial $p$, the cyclotomic polynomial $\Phi_p(x)$, and any PPI-polynomial $g(x)\in\Z[x]$ such that $g(b)=p$ for some $b\geq \mb(g)$. This collection, which includes all $f_p^{(b)}(x)$ for $b\geq 1$,	will be denoted by $\cF_p$ and called the {\em family of irreducible polynomials associated with $p$}.
\end{definition}

\begin{theorem}\label{partone} Given any positive rational prime  $p$, the family $\cF_p$ of irreducible polynomials associated with $p$ can be generated from the seed polynomial $\Phi_p(x)$ by the following rules:
	\begin{enumerate}
		\item The ascent procedure described in Proposition~\ref{seed};
		\item	Replacing selected coefficients $a_i$ of a polynomial $f_p^{(b)}(x)$ by the linear digit $x-(b-a_i)$. Note that this can be done for finitely many zero coefficients as well, including those that are associated with $x^n$ 	where $n>\deg f_p^{(b)}(x)$.
	\end{enumerate}
\end{theorem} 

\begin{proof}
	Proposition~\ref{seed} and the last example show us how to obtain all $f_p^{(b)}(x)$ for $b\geq 2$ from $f_p^{(1)}(x)=\Phi_p(x)$. Let $g(x)$ be any PPI-polynomial such that $g(b)=p$ for some $b\geq \mb(g)$. If all coefficients of $g$ are nonnegative, then $g(x)=f_p^{(b)}(x)$, and we are done. If at least one coefficient of $g$ is negative, then we write $g(x)=\sum b_i(x)x^i$ in base $x$, and replace all $b_i(x)$ by $b_i(b)$. The resulting polynomial, $f(x)$, also satisfies the property $f(b)=p$ as well as $b\geq \mb(f)$ (some $(x-a_i)$'s in $g(x)$ become $(b-a_i)$'s, where $b>b-a_i$). Moreover, its coefficients are nonnegative, which makes $f=f_p^{(b)}$. Therefore, any PPI-polynomial $g\in\cF_p$  can be obtained from some $f_p^{(b)}(x)$ by the reverse procedure: (i)~splitting the  coefficients of $f_p^{(b)}$ into two finite subsets, indexed by $I$ and $J$, and (ii)~replacing all $a_i$ with $i\in I$ with the linear digits $x-(b-a_i)$. It is not guaranteed that each such procedure will always give us an irreducible (hence, PPI) polynomial. For example, we cannot change {\em all} coefficients of the base-1 polynomial  $f_p^{(1)}(x)=\Phi_p(x)$, which would give us a factor of $x-(p-1)$, nor can we replace the coefficient 1 in $f_p^{(p)}(x)=x$ with the factor $x-(p-1)$. However, we can check whether $g(x)$ is   irreducible or not by applying Theorem~\ref{factoralg}. 
\end{proof}

\begin{corollary}\label{usual}
	For each positive rational prime $p$, the set $\cF_p$ is infinite, and always contains
	\[ \{  x+1,\, x,\, x-1,\, x-2,\, x-3,\,\dots\, ,p\} .
	\]	
\end{corollary}

\begin{proof}
	For bases $b>p$, the polynomial representative of $p$ is $p$ itself, and we may replace it with the linear polynomial	$g(x)=x-(b-p)$.  In addition, for $b=p$ we have 	$x$, and for $b=p-1$ we have $x+1$.
\end{proof}

\begin{example}
	Let us compute some of the elements of $\cF_2$ of degree $\leq 2$. We start with the set in Corollary~\ref{usual}: 
	\[ \{  x+1,\, x,\, x-1,\, x-2,\, x-3,\,\dots\, ,2\}.
	\]
The distinct base polynomials for $p=2$ are $2$ ($b\geq 3$), $x$ ($b=2$), and $x+1$ ($b=1$). The constant polynomial $2$ has already been replaced by $x-(3-2)=x-1$, $(x-(4-2)=2$, etc.	What if we also include the zero coefficient  in $0x+2$? Here are some of the additional polynomials to be considered:
\[\begin{array}{|c|c|c|c|c|} \hline b &   \mbox{0-replacement} & \mbox{2-replacement} & \mbox{result} & \mbox{PPI?}\\ \hline  
3 & x-(3-0)   & 2 & x^2-3x+2 & \mbox{no} \\ \hline 
3 &   x-(3-0) & x-(3-2) & x^2-2x-1  &  \mbox{yes}  \\
\hline 
4 & x-(4-0) & 2 &  x^2-4x+2 &   \mbox{yes} \\ \hline 
4 & x-(4-0) & x-(4-2) &x^2-3x-2 & \mbox{no}\\ \hline
5 & x-(5-0) & 2 &x^2-5x+2 & \mbox{no}\\ \hline
5 & x-(5-0) &  x-(5-2)  &x^2-4x-3 & \mbox{yes}\\ \hline
\vdots & \vdots & \vdots & \vdots& \vdots\\ \hline
\end{array}
\]
There are potentially infinitely many quadratics in $\cF_2$ to be obtained in this manner. 
Next, we consider the base polynomial $x=x+0$. We have observed that we cannot replace the coefficient 1 by $x-(2-1)$ and still obtain an irreducible polynomial. But if we put the constant term 0 into the mix, then we obtain
\[ (x-(2-1))x+(x-(2-0))=x^2-2\;\;\mbox{and}\;\; x+(x-(2-0))=2x-2,
\]
where the first polynomial is PPI and the second is not. For the last base polynomial, $x+1$, we cannot change both coefficients but can try to do so one at a time: we have
\[ (x-(1-1))x+1=x^2+1\;\;\mbox{and}\;\; x+(x-(1-1))=2x,
\]
	where once again $x^2+1$ is PPI but $2x$ is not. Technically, we cannot include $x^2+1$ in $\cF_2$, since we need to substitute $b=1<\mb(x^2+1)=2$. Hence, we have the following polynomials of degree $\leq 2$ in $\cF_2$, where the unknown quadratic polynomials are of the form $(x-b)x+2$ or $(x-b)x+(x-(b-2))$ for $b\geq 6$:
\[ \{ x^2-2,\, x^2-2x-1,\, x^2-4x+2,\,   x^2-4x-3,\dots,        x+1,\, x,\, x-1,\, x-2,\, x-3,\,\dots\, ,2\}.
\]	
The degree of polynomials in $\cF_2$, or in any $\cF_p$, do not {\em a priori} have an upper bound, as we are allowed to change the zero coefficients of large powers of $x$ in a base polynomial. For example, for $b=6$, we have
\[ (x-(6-0))x^2+2=x^3-6x^2+2\in\cF_2.\]

	%The only interesting non-base PPI-polynomials in $\cF_3$, if any,  will be obtained from $f_3^{(2)}(x)=x+1$ by replacing only one coefficient with $(x-(2-1))=(x-1)$: we have
	%\[ (x-1)x+1=x^2-x+1\;\;\mbox{and}\;\; x+(x-1)=2x-1.
	%\]
	%Both are PPI, and we include them in the final count:
	%\[ \cF_3=\{ 3,\, x^2+x+1,\, x^2-x+1,\, 2x-1,\, x+1,\, x,\, x-1,\, x-2,\, x-3,\dots\}       .
%	\]
%	We leave it to the reader to check that
%	\bea \cF_5&=& \{  5,\, x^4+x^3+x^2+x+1,\, x^3-x^2+1,\,  x^2+x-1,\, x^2+1,\,  x^2-2x+2,\\
%	&& \; x^2-3x+1,\,    2x-1,\,  2x-3,\,  x+2,   \,
%	x+1,\, x,\, x-1,\, x-2,\, x-3,  \dots                     \}      .
%	\eea
\end{example}

Finally, we have the following summary. Part~(2) especially shows how restrictive $\cF_p$ is, although it is infinite. 
\begin{theorem}\label{BHequal}
	Let $p$ and $q$ be positive rational primes. Then 
	\begin{enumerate}
		\item $\cF_p=\cF_q$ if and only if $p=q$.
		\item Given any prime $p$, there exist infinitely many PPI-polynomials $f$ that attain  the value $p$ at some  integer but no translate $f(x-b)$, $b\in\Z$, is in $\cF_p$.
		\item  (BH) holds if and only if every PPI-polynomial $f$ belongs to $\cF_p\,$ for some prime $p$.
	\end{enumerate}
\end{theorem}

Theorem~\ref{partone} gives us the finitely many rules of membership in $\cF_p$, which may be helpful in studying (BH). 

\begin{proof} As the only constant polynomial in $\cF_p$ is $p$ itself,  the first part is straightforward. For Part~(2), consider the irreducible polynomial $f(x)=x^m+p$ for any $m\geq 2$, which is PPI but not in $\cF_p$. The value $p$ is attained only at $x=0$, and any translate $g(x)=f(x-b)$ would attain the same value only at $x=b$, which needs to be positive for eligibility. Let us fix $b\geq 1$ and $m\geq 2$. Since 
	$ g(x)=x^m-mb\, x^{m-1}+O(x^{m-2}) 
	$, 
	we have 
	$ b< mb\leq H(g)\leq \mb(g) 
	$. 	
	As for Part~(3), any PPI-polynomial $f$ would have infinitely many prime values past $x=\mb(f)$ under (BH). The converse is also clear by the equivalent statement to Conjecture~\ref{bh}.
\end{proof}

%==========================================================

\section{Conclusion.}

%==========================================================

Rational primes and irreducible polynomials with integer coefficients serve the same purpose of being multiplicative building blocks in their respective structures. Moreover, both integers and polynomials have unique representations in various bases that mirror each other: the base-$x$ representation of polynomials not only highlights the analogy, but also allows us to uniquely associate to each prime a finitely generated family of irreducible polynomials that attain the prime value. We hope that the results and techniques in the article will make a new generation of mathematicians take a closer look at the Bouniakowsky Hypothesis with fresh eyes. Considering that it is a generalization of Dirichlet's Theorem, though, we suggest that the proof not be assigned as homework; not every student can be a George Dantzig~\cite{CJW}!

%==========================================================

\appendix
\section{Appendix: Arithmetic in base $x$.}

%==========================================================

\label{app}

The rule for {\bf addition} of digits in base $x$ is given by Table~2, where each carried digit is at most 1 (indicated by $(c)$):
%\begin{table}[h] 
	\[ \begin{array}{|c|c|c|c|}\hline   i,j\geq 0& i>j\geq 0 & j\geq i> 0\; (c) & i,j>0 \; (c)  \\ \hline
	[(i)]_x & [(x-i)]_x & [(x-i)]_x & [(x-i)]_x\\  
	{}  [(j)]_x & [(j)]_x & [(j)]_x & [(x-j)]_x\\ \hline 
	[(i+j)]_x & [(x-(i-j))]_x & [(1)(j-i)]_x & [(1)(x-(i+j))]_x\\ \hline
	\end{array}\] 
\begin{center} {\small {\bf Table 2.} Addition of digits in base $x$}\end{center}
	%\caption{Addition of digits in base $x$}\label{addition}
%\end{table}

%\begin{example}
	Consider the  positive polynomials
	\[ f(x)=2x^3-x^2+5x-6=[ (1)(x-1)(4)(x-6)]_x                   
	\]
	and
	\[ g(x)=x^3-x-1 =[  (0)(x-1)(x-2)(x-1)          ]_x.
	\]
	Direct polynomial addition and conversion into base $x$ gives us
	\[ f(x)+g(x)=3x^3-x^2+4x-7=[ (2)(x-1)(3) (x-7)              ]_x.
	\]
	Digital addition is equally easy:
	\[ \begin{array}{ccccc} \mbox{\em carry} & 1 & 1 & 1 & {} \\ \hline
	{} & (1) & (x-1) & (4) & (x-6)  \\
	+ & (0) & (x-1) & (x-2) & (x-1) \\ \hline
	{} & (2) & (x-1) & (3) & (x-7)
	\end{array}\]
%\end{example}
The rules for {\bf subtraction} of digits, given in Table~3 below,  similarly require at most a borrowed digit of 1 from the left (indicated by $(b)$), which amounts to adding $x$ to the top digit. The analogy is borrowing 1 from the left  and adding 10, the base, to the digit that comes up short in ordinary decimal subtraction. If a nonzero digit is to be subtracted from $(0)$, which may itself be preceded by a string of $(0)$'s, then we borrow $(1)$ from the first nonzero digit to the left of the string of zeros, send it down the line,  and never have to borrow from that lender again, as $x$ is larger than any digit. In short, any digit lends 1 to the right at most once.
%\begin{table}[h]
	\[\begin{array}{|c|c|c|}  \hline
	i\geq j\geq 0 & j>i\geq 0 \; (b)  & i>0,j\geq 0\\ \hline 
	[(i)]_x & [(1)(i)]_x & [(x-i)]_x \\
	{}  [(j)]_x & [(j)]_x & [(j)]_x  \\ \hline
	[(i-j)]_x & [(x-(j-i))]_x & [(x-(i+j))]_x \\ \hline \hline
 i\geq 0, j>0 \; (b) & j\geq i>0 & i>j>0\; (b)\\ \hline 
 [(1)(i)]_x & [(x-i)]_x & [(1)(x-i)]_x \\
{} [(x-j)]_x & [(x-j)]_x & [(x-j)]_x \\ \hline
	 [(i+j)]_x & [(j-i)]_x & [(x-(i-j))]_x\\ \hline
	\end{array}\]
\begin{center} {\small {\bf Table 3.} Subtraction of digits in base $x$}\end{center}	
	%\caption{Subtraction of digits in base $x$}\label{subtraction}
%\end{table}

%\begin{example}
	Let us subtract the same two polynomials directly, noting that $f>g$:	
	\[ f(x)-g(x)=x^3-x^2+6x-5=[(x-1)(5)(x-5)    ]_x.
	\]	
	Contrast with digital subtraction:
	\[ \begin{array}{ccccc} \mbox{\em borrow} & -1  & -1  & -1  &   \\ \hline
	{} & (1) & (x-1) & (4) & (x-6)  \\
	- & (0) & (x-1) & (x-2) & (x-1) \\ \hline
	{} & (0) & (x-1) & (5) & (x-5)
	\end{array}\]
%\end{example}
{\bf Multiplication} in base $x$  is done one digit at a time, followed by staggered addition. Table~4 shows the rules for digit multiplication.
%\begin{table}[h]
	\[\begin{array}{|c|c|c|}\hline  i,j\geq 0 & i,j>0 & i,j>0 \\ \hline 
	[(i)]_x & [(x-i)]_x & [(x-i)]_x \\
	{}[(j)]_x & [(j)]_x & [(x-j)]_x\\ \hline
	[(ij)]_x & [(j-1)(x-ij)]_x & [(x-(i+j))(ij)]_x\\ \hline 
	\end{array}
	\]
\begin{center} {\small {\bf Table 4.} Multiplication of digits in base $x$}\end{center}
%	\caption{Multiplication of digits in base $x$}\label{multiplication}
%\end{table}

%\begin{example}
	For the same polynomials $f$ and $g$ as in the last two examples, we find	
	\bea f(x)g(x)&=& 6+x-4x^2-7x^3+3x^4-x^5+2x^6\\
	&=& [(1)(x-1)(2)(x-8)(x-4)(1)(6)                       ]_x          .                  
	\eea
	The first one-digit multiplication is done as follows (the carry digits on top are to be added to the product, as is usual):
	\[ \begin{array}{cccccc}{} & {} & (x-2)  & (4)  & (x-7)  &  {} \\ \hline
	{} &{} & (1) & (x-1) & (4) & (x-6)  \\
	\times & {}  & {} & {} & {} & (x-1) \\ \hline
	{} &(1) & (x-3) & (5) & (x-11) & (6)
	\end{array}\]
	The second one-digit multiplication is
	\[ \begin{array}{cccccc}{} & {} & (x-3)  & (4)  & (x-8)  &  {} \\ \hline
	{} &{} & (1) & (x-1) & (4) & (x-6)  \\
	\times & {}  & {} & {} & {} & (x-2) \\ \hline
	{} &(1) & (x-5) & (6) & (x-16) & (12)
	\end{array}\]
	Multiplication by the last digit of $g$ is the same as the first. Finally, we stagger the results and add them:
	\[ \begin{array}{cccccccc} {} & {} & {} & {} & (1) & (x-1) & (4) & (x-6)  \\
	\times  & {} & {} & {} & {} & (x-1) & (x-2) & (x-1) \\ \hline
	{} & {} & {} & (1)  & (x-3) & (5) & (x-11) & (6)\\
	{} & {} &(1) & (x-5) & (6) & (x-16) & (12) & (0) \\
	+ & (1)  & (x-3) & (5) & (x-11) & (6)& (0)& (0) \\ \hline
	{} & (1) & (x-1) & (2) & (x-8) & (x-4) & (1) & (6)
	\end{array}\]
%\end{example}

{\bf Division by a monic polynomial} is always possible in $\Z[x]$. (See the Division Algorithm, Theorem~\ref{divalg}.) For example, for
%\begin{example} \label{ex43} For
	\[ f(x)=2x^4-5x^3+7x-1=[(1)(x-5)(0)(6)(x-1)]_x
	\]
	and
	\[ g(x)=x^2+x-3= [(1)(0)(x-3)]_x   ,
	\]	
	we have the long-division quotient and remainder
	\[ q(x)= 2x^2-7x+13   
	\;\;\mbox{and}\;\; r(x)= -27x+38,
	\]	
where $r(x)$ is not   positive, but its degree (resp., its ``absolute value'') is still strictly less than the degree of $g(x)$ (resp., $g(x)$ itself). Hence, we may subtract 1 from the quotient and add $g(x)$ to the remainder, which will immediately give us a positive expression. The adjusted quotient and remainder are then  
	\[ q(x)= 2x^2-7x+13-1   =2x^2-7x+12=[(1)(x-7)(12)]_x
	\]
	and
	\[0\leq  r(x)=(-27x+38)+(x^2+x-3)   = x^2-26x+35= [(x-26)(35)]_x <g(x).
	\]
	Digital division yields the result directly, as shown below:
	\[\begin{array}{ll} {} &\;\;\;\;\;\;\;\;\;\;\;\;\;\;\;\;\;\;\;\;\;\;\, (1)\;\;\; (x-7)\;\;\;\, (12)\\
	%(1)\;(0)\;(x-3) &\showdiv{(1)\; (x-5)\;\; (0)\;\;\;\; (6)\; (x-1)}\\ 
	\cline{2-2} (1)\;(0)\;(x-3) \; )\!\!\!\!\! & \;\,  (1)\; (x-5)\;\;\;\;\, (0)\;\;\;\;\;\; (6)\;\; \;\;\;\,(x-1) \\
	{} & \;\, (1)\;\;\;\; (0) \;\;\;\;\, (x-3) \; \;\; (0)\;\;\;\;\;\;\;\;\, (0)\\   \cline{2-2} 
	{} & \;\, (0) \; (x-6)\;\;\;\;\, (3)\;\;\;\;\;\; (6)\;\;\;\;\;\, (x-1)\\
	{} & \;\;\;\;\;\;\; (x-7)\; (x-10)\; (21) \;\;\;\;\;\;\;\, (0)\\ \cline{2-2}
	{} & \;\;\;\;\;\;\; \;\;\;\, (0)\;\;\;\;\;\;\; (12)\;\,  (x-15)\; (x-1)\\ 
	{} & \;\;\;\;\;\;\; \;\;\;\;\;\;\;\;\;\;\; \;\;\;\; (12)\;\;\;\;\, (11)\;\;\;\; (x-36)\\  \cline{2-2} 
	{} & \;\;\;\;\;\;\; \;\;\;\;\;\;\;\;\;\;\; \;\;\;\;\;  (0)\;\; (x-26)\;  \;\;\,  (35)
	\end{array}
	\] 

\end{document}